\documentclass[11pt]{amsart} 

\usepackage{amsfonts,amsmath,amssymb,amscd,amsthm}
\usepackage[english]{babel} 
\usepackage[all]{xy}
\usepackage[backref, colorlinks, linktocpage, citecolor = blue, linkcolor = blue]{hyperref}
\usepackage{enumitem} 

\newtheorem{theorem}{Theorem}
\newtheorem*{maintheorem*}{Main Theorem}
\newtheorem{lemma}[theorem]{Lemma}
\newtheorem{corollary}[theorem]{Corollary}
\newtheorem{proposition}[theorem]{Proposition}

\newtheorem*{question*}{Question}
\newtheorem*{lem*}{Lemma}

\theoremstyle{definition}

\theoremstyle{remark}
\newtheorem{remark}[equation]{Remark}
\newtheorem*{remark*}{Remark}
\newtheorem{example}[equation]{Example}

\DeclareMathOperator{\Aut}{Aut}
\DeclareMathOperator{\Pic}{Pic}
\DeclareMathOperator{\CH}{CH}
\DeclareMathOperator{\Cent}{Cent}

\DeclareMathOperator{\Tr}{Tr}
\DeclareMathOperator{\Spec}{Spec}

\DeclareMathOperator{\Hom}{Hom}

\DeclareMathOperator{\grad}{gr}
\DeclareMathOperator{\rank}{rank}

\renewcommand{\AA}{\mathbb A}

\newcommand{\CC}{\mathbb C}
\newcommand{\Cst}{\CC^{*}}
\newcommand{\ZZ}{\mathbb Z}
\newcommand{\QQ}{\mathbb Q}
\newcommand{\NN}{\mathbb N}
\newcommand{\GG}{\mathbb G}
\newcommand{\Zpn}{(\ZZ / p\ZZ)^n}

\newcommand{\OO}{\mathcal O}
\newcommand{\VVV}{\mathcal V}
\newcommand{\SSS}{\mathcal S}
\newcommand{\Sn}{\SSS_{n}}

\newcommand{\mm}{\mathfrak m}

\newcommand{\simto}{\overset{\sim}{\rightarrow}}
\newcommand{\into}{\hookrightarrow}

\newcommand{\ps}{\par\smallskip}

\title[Is the affine space determined by its automorphism group?]
{Is the affine space determined by its automorphism group?}
\author[H. Kraft \and A. Regeta \and I. van Santen]
{Hanspeter Kraft \and Andriy Regeta \and 
Immanuel van Santen (born Stampfli)}

\address{Departement Mathematik und Informatik, 
Universit\"at Basel,\newline
\indent Spiegelgasse 1, CH-4051 Basel, Switzerland}
\email{Hanspeter.Kraft@unibas.ch}

\address{Mathematisches Institut, 	
Universit\"at zu K\"oln,\newline
\indent Weyertal 86-90, 50931 K\"oln, Germany}
\email{andriyregeta@gmail.com}

\address{Departement Mathematik und Informatik, 
Universit\"at Basel,\newline
\indent Spiegelgasse 1, CH-4051 Basel, Switzerland}
\email{immanuel.van.santen@math.ch}

\begin{document}



\begin{abstract}
	In this note we study the problem of characterizing the complex affine space
	$\AA^n$ via its automorphism group. We prove
	the following. Let $X$ be an irreducible quasi-projective $n$-dimensional variety
	such that $\Aut(X)$ and $\Aut(\AA^n)$ are isomorphic as abstract groups. 
	If $X$ is either quasi-affine and toric or $X$ is smooth with Euler characteristic $\chi(X) \neq 0$
	and finite Picard group $\Pic(X)$, 
	then $X$ is isomorphic to $\AA^n$.
	
	The main ingredient is the following result.
	Let $X$ be a smooth irreducible quasi-projective
	variety of dimension $n$ with finite $\Pic(X)$.
	If $X$ admits a faithful $(\ZZ / p \ZZ)^n$-action for a prime $p$ and $\chi(X)$ is not
	divisible by $p$, then the identity component of the centralizer 
	$\Cent_{\Aut(X)}( (\ZZ / p \ZZ)^n)$ is a torus.
\end{abstract}

\maketitle

\section{Introduction}
In 1872, Felix Klein suggested as part of his Erlangen Programm
to study geometrical objects through their symmetries.
In the spirit of this program it is natural to ask to which
extent a geometrical object is determined by its automorphism group.
For example, a smooth manifold, a symplectic manifold 
or a contact manifold is determined by its automorphism group, see 
\cite{Fi1982Isomorphisms-betwe,Ry1995Isomorphisms-betwe, 
Ry2002Isomorphisms-betwe}.

We work over the field of complex numbers $\CC$.
For a variety $X$ we denote by $\Aut(X)$ the group of regular automorphisms 
of $X$. As the automorphism group of a variety is usually quite small, 
it almost never determines the variety. 
However, if $\Aut(X)$ is large, like for the affine space $\AA^{n}$, 
this might be true. Our guiding question is 
the following.

\begin{question*} 
	Let $X$ be a variety. Assume that 
	$\Aut(X)$ is isomorphic to the group $\Aut(\AA^n)$. 
	Does this imply that $X$ is isomorphic to the affine 
	space $\AA^n$?
\end{question*}

This question cannot have a positive answer for \emph{all}
varieties $X$. For example, $\Aut(\AA^n)$ and 
$\Aut(\AA^n \times V)$ are isomorphic for any 
complete variety $V$ with 
a trivial automorphism group. Similarly, the automorphism group
of $\AA^n$ does not change if one forms the disjoint union with
a variety with a trivial automorphism group.
Thus we have to impose certain assumptions on $X$.
Moreover, we assume that $n \geq 1$,
since there exist many affine varieties 
with a trivial automorphism group.

In \cite{Kr2017Automorphism-Group}, 
the first author studies the problem of characterizing the 
affine space $\AA^n$ by its automorphism group
regarded as a so-called ind-group. It is shown that if 
$X$ is a connected affine variety such that $\Aut(X)$ and $\Aut(\AA^n)$ 
are isomorphic
as ind-groups, then $X$ and $\AA^n$ are isomorphic as 
varieties. For some generalizations of this result we refer to
\cite{Re2017Characterization-o}.

In dimension two, a generalization of our guiding question is studied in 
\cite{LiReUr2017Characterisation-o}. 
For an irreducible normal affine surface $X$ it is shown 
that if $\Aut(X)$ is isomorphic to $\Aut(Y)$ for an affine toric surface $Y$, 
then $X$ is isomorphic to $Y$.

In order to state our main result, let us introduce the following notation. For a variety $X$
we denote by $\chi(X)$ the Euler characteristic and by $\Pic(X)$ the Picard group.

\begin{maintheorem*}
	\label{thm:AcyclicAndToric}
	Let $X$ be an irreducible quasi-projective variety such that 
	$\Aut(X) \simeq \Aut(\AA^{n})$. 	
	Then $X\simeq \AA^{n}$ if one of the following conditions holds.
	\begin{enumerate}[label = $(\arabic*)$]
		\item \label{thm:AcyclicAndToric:Acyclic} 
		$X$ is smooth,
		$\chi(X) \neq 0$, $\Pic(X)$ is finite, and $\dim X\leq n$;		
		\item \label{thm:AcyclicAndToric:Toric} 
		$X$ is toric, quasi-affine, and $\dim X \geq n$.
	\end{enumerate}
\end{maintheorem*}

As a direct application of this result we get that $\Aut(\AA^n \setminus S)$
and $\Aut(\AA^n)$ are non-isomorphic as abstract groups 
for every non-empty closed subset $S$ in $\AA^n$ with Euler characteristic $\chi(S) \neq 1$.

Let us give an outline of the proof.
Let $\theta \colon \Aut(\AA^n) \simto \Aut(X)$ be an isomorphism.
First, we prove that if a maximal torus of $\Aut(\AA^{n})$ is mapped onto an 
algebraic group via $\theta$ and $X$ is quasi-affine, then $X \simeq \AA^{n}$
(see Proposition~\ref{prop:XandCnIsomorphic}).
Our main result to achieve this condition is the following.

\begin{theorem}
	\label{thm:MainThm}
	Let $X$ and $Y$ be irreducible quasi-projective varieties such that 
	$\dim Y \leq \dim X=:n$. Assume that the following conditions are satisfied:
	\begin{enumerate}[label = $(\arabic*)$]
	\item $X$ is quasi-affine and toric;	
	\item $Y$ is smooth, $\chi(Y) \neq 0$, and $\Pic(Y)$ is finite. 
	\end{enumerate} 
	If $\theta\colon \Aut(X)\simto\Aut(Y)$ is an isomorphism, then $\dim Y=n$, and 
	for each $n$-dimensional torus $T \subseteq \Aut(X)$, the identity component of
	the image $\theta(T)^\circ$ is a closed torus of dimension $n$ in $\Aut(Y)$.
	Furthermore, $Y$ is quasi-affine.
\end{theorem}

For the definition of the topology on $\Aut(X)$, the definition of the identity component $G^\circ$
of a subgroup $G \subseteq \Aut(X)$ and the definition of algebraic subgroups of $\Aut(X)$
we refer to section~\ref{sec.alg_subgroups}.

For the proof of Theorem~\ref{thm:MainThm} we first remark that
every torus $T\subseteq \Aut(X)$ of maximal dimension $n = \dim X$ is self-centralizing (Lemma~\ref{lem:centralizer_of_torus}). 
For any prime $p$
the torus $T$ contains a unique subgroup $\mu_p$ which is isomorphic to 
$(\ZZ / p \ZZ)^n$. In particular, $T \subseteq \Cent_{\Aut(X)}(\mu_p)$,
and thus the image of $T$ under an isomorphism $\theta \colon \Aut(X) \to \Aut(Y)$
is mapped to a subgroup of the centralizer of $\theta(\mu_p)$
inside $\Aut(Y)$. Our strategy is then to prove that the identity component
of the centralizer $\Cent_{\Aut(Y)}(\theta(\mu_p))$ is an algebraic group.
Our main result in this direction is the following generalization of
\cite[Proposition~3.4]{KrSt2013On-automorphisms-o}.

\begin{theorem}
	\label{thm:Centralizer}
	Let $X$ be a smooth, irreducible, quasi-pro\-jective variety of dimension $n$ with finite
	Picard group $\Pic(X)$. Assume that $X$ is endowed with a faithful $(\ZZ / p \ZZ)^n$-action
	for some prime $p$ which does not divide $\chi(X)$.
	Then  the centralizer $G := \Cent_{\Aut(X)}((\ZZ / p\ZZ)^n)$ 
	is a closed subgroup of $\Aut(X)$, and the identity component $G^\circ$ is a closed torus 
	of dimension $\leq n$.		
\end{theorem}

For the proof of Theorem~\ref{thm:Centralizer} we use the fact that $p$
does not divide $\chi(X)$ to find a fixed point of the 
$(\ZZ / p \ZZ)^n$-action on $X$ (Proposition~\ref{prop:EulerChar}), and the smoothness of $X$ 
to show that the fixed point variety $X^{(\ZZ / p \ZZ)^n}$ is finite.

{\small
\subsection*{Acknowledgements} 
We would like to thank 
Alvaro Liendo for fruitful discussions.}

\par\medskip
\section{Preliminary results}

Throughout this note we work over the field $\CC$ of complex numbers.
A variety will be a reduced separated scheme of finite type over $\CC$.

\subsection{Quasi-affine varieties}
Let us recall some well-known results about quasi-affine varieties.
The first lemma is known for affine varieties and can be reduced to this case
by using open affine coverings.

\begin{lemma}
	\label{lem:RegFunctionsOnAProduct}
	Let $X$, $Y$ be varieties. Then the natural homomorphism
	\[
		\OO(X) \otimes_\CC \OO(Y) \to \OO(X \times Y)
	\] 
	is an isomorphism of $\CC$-algebras.
\end{lemma}

\begin{lemma}
	\label{lem:OpenImmersion}
	Let $X$ be a quasi-affine variety. Then the canonical morphism
	$\eta \colon X \to \Spec \OO(X)$ is a dominant open immersion 
	of schemes.
\end{lemma}

\begin{proof}
	Let $f \colon \Spec \OO(X) \to \AA^1$ be a morphism which vanishes
	on $\eta(X)$. Since $f$ can be understood as a regular function 
	$\tilde{f}$ on $X$, we get the following commutative diagram
	\[
		\xymatrix{
			& \AA^1 \\
			X \ar[ru]^-{\tilde{f}} \ar[r]_-{\eta} & \Spec \OO(X) \ar[u]_-{f}
		}
	\]
	which shows that $\tilde f=0$. 
	This implies that $\eta$ is dominant.
	
	We have an open immersion $\iota \colon X \to Y$ where $Y$ is affine, 
	and therefore a decomposition of $\iota$
	\[
		\xymatrix{
			\iota \colon X \ar[r]^-{\eta} & \Spec\OO(X) \ar[r]^-{\alpha} & Y \, .
		}
	\]
	In particular, $\eta$ is injective. 
	For any nonzero $f \in \OO(Y)$ such that $Y_{f}\subseteq\iota(X)$ we 
	see that $\iota$ induces an isomorphism $X_{\iota^{*}(f)} \simto Y_{f}$, 
	hence the composition
	\[
		\xymatrix{
			\iota' \colon X_{\alpha^{*}(f)} \ar[r]^-{\eta'} &  
			\Spec \OO(X)_{\alpha^{*}(f)} = \alpha^{-1}(Y_{f}) \ar[r]^-{\alpha'} 
			& Y_{f}
		}
	\]
	is an isomorphism where $\eta'$ and $\alpha'$ are the restrictions of 
	$\eta$ and $\alpha$ respectively.
	Since $\iota'$ is an isomorphism
	$X_{\alpha^\ast(f)}$ is affine and thus 
	$\eta'$ is an isomorphism, 
	because $\OO(X)_{\alpha^\ast(f)} = \OO(X_{\alpha^\ast(f)})$.
	Therefore, $\eta$ is a local isomorphism, hence an open immersion.
\end{proof}

\begin{lemma}
	\label{lem:ExtensionOfAction}
	Let $X$ be a quasi-affine variety and $Y$
	a variety. Then every morphism
	$Y \times X \to X$ extends uniquely to a morphism 
	$Y \times  \Spec \OO(X) \to \Spec \OO(X)$ via
	$X \to \Spec \OO(X)$. In particular, every regular
	action of an algebraic group on $X$ 
	extends to a regular action on $\Spec \OO(X)$.
\end{lemma}

\begin{proof}
	We can assume that $Y$ is affine.
	By Lemma~\ref{lem:RegFunctionsOnAProduct} 
	we have $\OO(Y \times X) = \OO(Y) \otimes_{\CC} \OO(X)$.
	Hence $Y \times X \to X$ induces
	a homomorphism of $\CC$-algebras 
	$\OO(X) \to \OO(Y) \otimes_{\CC} \OO(X)$
	which in turn gives the desired extension 
	$Y \times \Spec \OO(X) \to \Spec \OO(X)$.
\end{proof}

\subsection{Algebraic structure on the group of automorphisms}
\label{sec.alg_subgroups}
In this subsection, we recall some basic results
about the automorphism group $\Aut(X)$ of a
variety $X$. The survey \cite{Bl2016Algebraic-Structur}
and the article \cite{Ra1964A-note-on-automorp} will
serve as references.
Recall that a \emph{morphism} $\nu \colon A \to \Aut(X)$
is a map from a variety $A$ to $\Aut(X)$ such that the associated map
\[
	\tilde{\nu} \colon 
	A \times X \to X \, , \quad (a, x) \mapsto \nu(a)(x)
\] 
is a morphism of varieties. We get a topology on $\Aut(X)$ by
declaring a subset $F \subset \Aut(X)$ to be \emph{closed}, 
if the preimage $\nu^{-1}(F)$ under every morphism 
$\nu \colon A \to \Aut(X)$ is closed in $A$. Similarly, a \emph{morphism} 
$\nu = (\nu_1, \nu_2) \colon A \to \Aut(X) \times \Aut(X)$ is a map from
a variety $A$ into $\Aut(X) \times \Aut(X)$ such that $\nu_1$ and $\nu_2$ are morphisms. Thus we
get analogously as before a topology on $\Aut(X) \times \Aut(X)$. Note that for morphisms
$\nu, \nu_1, \nu_2 \colon A \to \Aut(X)$ the following maps are again morphisms
\begin{align*}
	A \to \Aut(X) \, , \ & a \mapsto \nu_1(a) \circ \nu_2(a) \\
	A \to \Aut(X) \, , \ & a \mapsto \nu(a)^{-1}
\end{align*}
and that $\nu^{-1}(\Delta)$ is closed in $A$ where $\Delta \subset \Aut(X) \times \Aut(X)$
denotes the diagonal. From these properties, one can deduce that $\Aut(X)$ behaves a bit 
like an algebraic group:

\begin{lemma}
	\label{lem:basic_properties_of_the_topology}
	For any variety $X$, the maps
	\begin{align*}
		\Aut(X) \times \Aut(X) \to \Aut(X) \, , \ & (\varphi_1, \varphi_2) \mapsto \varphi_1 \circ \varphi_2 \\
		\Aut(X) \to \Aut(X) \, , \ & \varphi \mapsto \varphi^{-1}
	\end{align*}
	are continuous and the diagonal $\Delta$ is closed in $\Aut(X) \times \Aut(X)$.
\end{lemma}

\begin{example}
	For any set $S \subseteq \Aut(X)$ the centralizer $\Cent(S)$
	is a closed subgroup of $\Aut(X)$. This is a 
	consequence of Lemma~\ref{lem:basic_properties_of_the_topology}.
\end{example}

For a subset $S \subseteq \Aut(X)$ its \emph{dimension} is defined by
\[
	\dim S := \sup\left\{d \ \Big| \  	\begin{array}{l}
		\textrm{there exists a variety $A$ of dimension $d$ and an in-} \\
	 	\textrm{jective morphism $\nu\colon A \to \Aut(X)$ with image in $S$}
	\end{array}\right\}.
\]
The following lemma generalizes the classical dimension estimate to morphisms $A \to \Aut(X)$.

\begin{lemma}
	\label{lem:dim_estimate}
	If $\nu \colon A \to \Aut(X)$ is a morphism, then $\dim \nu(A)Ê\leq \dim A$.
\end{lemma}

\begin{proof}
	Let $\eta \colon B \to \Aut(X)$ be an injective morphism such that $\eta(B) \subseteq \nu(A)$.
	The statement follows if we prove $\dim B \leq \dim A$.
	Since $\eta$ is injective, there exist points $x_1, \ldots, x_k \in X$ such that the map
	\[
		\eta' \colon B \to \underbrace{X \times \cdots \times X}_{\textrm{$n$-times}} \, , \quad 
		b \mapsto (\eta(b)(x_1), \ldots, \eta(b)(x_k))
	\]
	is injective, see e.g. \cite[Lemma~1]{Ra1964A-note-on-automorp}. Let
	\[
		\nu' \colon A \to \underbrace{X \times \cdots \times X}_{\textrm{$n$-times}} \, , \quad 
		a \mapsto (\nu(a)(x_1), \ldots, \nu(a)(x_k)) \, .
	\]
	Since $\eta(B) \subseteq \nu(A)$, we get $\eta'(B) \subseteq \nu'(A)$ and thus
	$\dim B = \dim \overline{\eta'(B)} \leq \dim \overline{\nu'(A)} \leq \dim A$.
\end{proof}

For a subgroup $G \subseteq \Aut(X)$, the \emph{identity component} $G^\circ \subseteq G$ is defined by
\[
	G^\circ = \left\{ \, g \in G \ \Big| \ 
	\begin{array}{l}
		\textrm{there exists an irreducible variety $A$ and a morphism} \\
	 	\textrm{$\nu\colon A \to \Aut(X)$ with image in $G$ 
		such that $g,e\in\nu(A)$}
	\end{array}
	\right\} \, .
\]
We call a subgroup $G \subseteq \Aut(X)$ \emph{connected} if $G = G^\circ$.
In the next proposition, we list several properties of the identity component of a subgroup of
$\Aut(X)$.

\begin{proposition}
	\label{prop:Id_component}
	Let $X$ be a variety and let $G \subseteq \Aut(X)$ be a subgroup. 
	Then the following holds:
	\begin{enumerate}[label=$(\arabic*)$]
		\item \label{prop:Id_component_1} $G^\circ$ is a normal subgroup of $G$;
		\item \label{prop:Id_component_2} The cosets of $G^\circ$ in $G$ are
			the equivalence classes under:
			\[
				g_1 \sim g_2  \ :\iff \ \left\{
				\begin{array}{l}
					\textrm{there exists an irreducible variety $A$} \\
					\textrm{and a morphism $\nu \colon A \to \Aut(X)$} \\
					\textrm{with image in $G$ such that $g_1, g_2 \in \nu(A)$} \, ; 
				\end{array}
				\right.
			\]
		\item \label{prop:Id_component_3} For each morphism $\nu \colon A \to \Aut(X)$ with
								      image in $G$, 
								      the preimage $\nu^{-1}(G^\circ)$
								      is closed in $A$. In particular, if
								      $G$ is closed in $\Aut(X)$, then
								      $G^\circ$ is closed in $\Aut(X)$;
		\item \label{prop:Id_component_4} If $X$ is quasi-projective and $G$ is closed in $\Aut(X)$, 
								      then the index of $G^\circ$ in $G$ is countable.
	\end{enumerate}
\end{proposition}

\begin{proof}
	\ref{prop:Id_component_1}: 
	One can directly see, that $G^\circ$ is a normal subgroup of $G$.

	\ref{prop:Id_component_2}: We first show that $``\sim"$
	defines an equivalence relation on $G$. Reflexivity and symmetry are obvious. 
	For proving the transitivity, let $g \sim h$ and $h \sim k$. 
	By definition there exist irreducible varieties $A$, $B$,
	morphisms $\nu \colon A \to \Aut(X)$, $\eta \colon B \to \Aut(X)$ with image 
	in $G$ and $a_1, a_2 \in A$, $b_1, b_2 \in B$ such that $\nu(a_1) = g$, 
	$\nu(a_2) = h$, $\eta(b_1) = h$, $\eta(b_2) = k$. Then 
	\[
		A \times B \to \Aut(X) \, , \quad (a, b) \mapsto \nu(a) \circ h^{-1} \circ \eta(b)
	\]
	is a morphism with image in $G$ that maps $(a_1, b_1)$ onto 
	$g$ and $(a_2, b_2)$ onto $k$.
	Thus $g \sim k$, which proves the transitivity. In particular, 
	$G^\circ$ is the equivalence class
	with respect to $\sim$ which contains the neutral element $e$. This implies
	the statement
	
	\ref{prop:Id_component_3}: Let
	\[
		\bigcup_{i=1}^k B_i = \overline{\nu^{-1}(G^\circ)} \subseteq A
	\]
	be the decomposition of the closure of $\nu^{-1}(G^\circ)$ into 
	irreducible components $B_1, \ldots, B_k$.
	Thus $B_i \cap \nu^{-1}(G^\circ)$ is non-empty. Since $\nu$ has image in $G$ it follows from  
	the transitivity of $``\sim"$ that $\nu(B_i) \subseteq G^\circ$. Thus 
	$B_i \subseteq \nu^{-1}(G^\circ)$ for all $i$. 
	Hence $\nu^{-1}(G^\circ)$ is closed in $A$.
	
	\ref{prop:Id_component_4}: 
	Let $\nu\colon A \to \Aut(X)$ be a morphism. Since $\nu^{-1}(G) \subseteq A$ is closed, 
	it has only finitely many irreducible components. This implies that its image 
	$\nu(A)$ meets only finitely many cosets of $G^{\circ}$ in $G$. 
	The claim follows if we show that there exist countably many morphisms of varieties into $\Aut(X)$ 
	whose images cover $\Aut(X)$. We will show this.

	Since $X$ is quasi-projective,
	there exists a projective variety $\overline{X}$ and an open embedding 
	$X \subseteq \overline{X}$.
	For each polynomial $p \in \QQ[x]$ we denote by 
	$\textrm{Hilb}^p$
	the Hilbert scheme of $\overline{X} \times \overline{X}$ 
	associated to the Hilbert polynomial $p$, and denote by 
	$\mathcal{U}^p \subseteq \textrm{Hilb}^p \times \overline{X} \times \overline{X}$
	the universal family, which is by definition flat over $\textrm{Hilb}^p$. 
	By \cite[Theorem~3.2]{Grothendieck:1995aa}, $\textrm{Hilb}^p$
	is a projective scheme over $\CC$. For $i=1, 2$ 
	consider the following morphisms
	\[
		q_i \colon (\textrm{Hilb}^p \times X \times X) \cap \mathcal{U}^p \to 
		\textrm{Hilb}^p \times X \, , \quad (h, x_1, x_2) \mapsto (h, x_i)
	\]
	which are defined over $\textrm{Hilb}^p$. 
	By \cite[Proposition~9.6.1]{Gr1966Elements-de-geomet-IV}, 
	the points $h \in H$ where the restriction
	\[
		q_i |_{\{h\}} \colon (\{ h \} \times X \times X) \cap 
		\mathcal{U}^p \to \{ h \} \times X
	\]
	is an isomorphism, form a constructible subset $S^p$ of $\textrm{Hilb}^p$. Now choose locally closed
	subsets $S_j^p$, $j=1, \ldots, k_p$ of $\textrm{Hilb}^p$ that cover $S^p$. 
	We equip each $S_j^p$ with the underlying 
	reduced scheme structure of $\textrm{Hilb}^p$. Note that 
	$(\textrm{Hilb}^p \times X \times X) \cap \mathcal{U}^p$ and 
	$\textrm{Hilb}^p \times X$ are both flat over $\textrm{Hilb}^p$.
	Therefore, we can apply \cite[Proposition 5.7]{1971Revetements-etales} 
	and we get that $q_i$ restricts to an isomorphism
	over $S_j^p$. Thus for each $j$ we get a morphism of varieties
	\[
		\xymatrix{
			S_j^p \times X \ar[rr]^-{(q_1 |_{S_j^p})^{-1}} &&
			(S_j^p \times X \times X) \cap \mathcal{U}^p \ar[rr]^-{q_2 |_{S_j^p}} &&
		 	S_j^p \times X \ar[r] & X
		}
	\]
	which defines a morphism $S_j^p \to \Aut(X)$. For each automorphism
	$\varphi$ in $\Aut(X)$, the closure in $\overline{X} \times \overline{X}$ 
	of the graph 
	$\Gamma_{\varphi} \subseteq X \times X$ defines a (closed) point in the 
	Hilbert scheme 
	$\textrm{Hilb}^p$ for a certain rational polynomial $p$,
	which belongs to $S^p$. Thus the images of the morphisms 
	$S_j^p \to \Aut(X)$ cover $\Aut(X)$. Since there are only countably 
	many rational polynomials, the claim follows.
%
\end{proof}

We say that $G$ is an \emph{algebraic subgroup} of $\Aut(X)$
if there exists a morphism $\nu \colon H \to \Aut(X)$ of an algebraic group $H$ onto $G$
which is a homomorphism of groups.

The next result gives a criterion for a subgroup of $\Aut(X)$
to be algebraic. The main argument is due to Ramanujam \cite{Ra1964A-note-on-automorp}.

\begin{theorem}
	\label{thm:CharacterizationAlgGrps}
	Let $X$ be an irreducible variety and let $G \subseteq \Aut(X)$
	be a subgroup. Then the following statements are equivalent:
	\begin{enumerate}[label = $(\arabic*)$]
		\item \label{thm:Characterization_G-alg} $G$ is an algebraic subgroup of $\Aut(X)$;
		\item \label{thm:Characterization_Constructible} there exists a 
			morphism of a variety into $\Aut(X)$ with image $G$;
		\item \label{thm:Characterization_dimG_and_index_finite} 
			$\dim G$ is finite and $G^\circ$ has finite index in $G$;
		\item \label{thm:Characterization_G-alg_Unique_alg_str}
			there is a unique structure of an algebraic group on $G$
			such that for each irreducible variety $A$ we have a bijective correspondence
			\[
				\left\{ 
					\begin{array}{c}
						\textrm{morphisms of} \\ 
						\textrm{varieties $A \to G$} 
					\end{array}
				\right\}
				\stackrel{1:1}{\longleftrightarrow} 
				\left\{ 
					\begin{array}{c}
						\textrm{morphisms $A \to \Aut(X)$} \\
						\textrm{with image in $G$}
					\end{array}
				\right\}
			\]
			given by $f \mapsto \iota \circ f$
			where $\iota \colon G \to \Aut(X)$ denotes the canonical inclusion.
	\end{enumerate}
\end{theorem}

\begin{proof}
 	The implication \ref{thm:Characterization_G-alg} $\Rightarrow$ \ref{thm:Characterization_Constructible}
	is obvious.
	
	Assume that \ref{thm:Characterization_Constructible} holds, i.e. there 
	is a morphism $\eta \colon A \to \Aut(X)$ with image equal to $G$.	
	By Lemma~\ref{lem:dim_estimate} we get $\dim G \leq \dim A$ 
	and hence $\dim G$ is finite. Since $A$ is a variety and thus has only 
	finitely many irreducible
	components, it follows from Proposition~\ref{prop:Id_component}~\ref{prop:Id_component_2} 
	that $G^\circ$ has finite index in $G$. This proves
	\ref{thm:Characterization_Constructible} 
	$\Rightarrow$ \ref{thm:Characterization_dimG_and_index_finite}.
	
	The implication \ref{thm:Characterization_dimG_and_index_finite} $\Rightarrow$
	\ref{thm:Characterization_G-alg_Unique_alg_str} is done in 
	\cite[Theorem~p.26]{Ra1964A-note-on-automorp} in case $G = G^\circ$
	for irreducible $A$. Thus in the general
	case, $G^\circ$ carries the structure of an algebraic group with the required property. Since $G^\circ$
	has finite index in $G$ we obtain a unique structure of an algebraic group on $G$ extending the 
	given structure on $G^\circ$. It remains to see that the required property holds for $G$.
	Note that the canonical inclusion $\iota \colon G \to \Aut(X)$ is a morphism and thus each morphism
	of varieties $A \to G$ yields a morphism $A \to \Aut(X)$ by composing with $\iota$.
	For the reverse, let $\nu \colon A \to \Aut(X)$ be a morphism with image in $G$. Since
	$A$ is irreducible, by Proposition~\ref{prop:Id_component}~\ref{prop:Id_component_2} 
	there is $g \in G$ such that the image of $\nu$ lies in $g G^\circ$.
	Note that the composition of $\nu$ with 
	$\theta_{g^{-1}} \colon \Aut(X) \to \Aut(X)$, $\varphi \mapsto g^{-1} \circ \varphi$ is a morphism
	with image in $G^\circ$. Thus $\theta_{g^{-1}} \circ \nu$
	corresponds to a morphism $A \to G^\circ$ of varieties.
	Using that $G \to G$, $h \mapsto g h$ is an isomorphism of varieties, 
	we get that $\nu$
	corresponds to a morphism $A \to G$ of varieties. This proves
	\ref{thm:Characterization_dimG_and_index_finite} $\Rightarrow$
	\ref{thm:Characterization_G-alg_Unique_alg_str}.
	
	The implication \ref{thm:Characterization_G-alg_Unique_alg_str} $\Rightarrow$
	\ref{thm:Characterization_G-alg} is obvious.
\end{proof}

\subsection{Ingredients from toric geometry}

Recall that a toric variety is a normal irreducible variety $X$ together
with a regular faithful action of a torus of dimension $\dim X$.  For
details concerning toric varieties we refer to
\cite{Fu1993Introduction-to-to}.
The first two lemmas are certainly well-known;
for the convenience of the reader we present for both a short proof.

\begin{lemma}	
	\label{lem:centralizer_of_torus}
	Let $X$ be a toric variety, and let $T$ be a torus of dimension
	$\dim X$ that acts faithfully on $X$. Then 
	the centralizer of $T$ in $\Aut(X)$ is equal to $T$.
	In particular, the image of $T$ in $\Aut(X)$ is closed.
\end{lemma}

\begin{proof}
	Let $g \in \Aut(X)$ such that $g t = t g$ for all $t \in T$.
	By definition, there is an open dense $T$-orbit in $X$, say $U$.
	Since $g U \cap U$ is non-empty, there exists $x \in U$ such that $g x \in U$.
	Using that $U = Tx$ we find $t_0 \in T$ with $gx = t_0x$. Thus for
	each $t \in T$ we get
	\[
		g t x = t g x = t t_0 x = t_0 t x \, .
	\]
	Using that $U = Tx$ is dense in $X$, we get $g = t_0$.
\end{proof}

\begin{lemma}
	\label{lem:ToricCoordinateRing}
	Let $X$ be a toric variety. Then the coordinate ring $\OO(X)$
	is finitely generated and integrally closed.
\end{lemma}

\begin{proof}
	This follows from \cite{Kn1993Uber-Hilberts-vier}. Here is a short direct proof.	
	Let $N$ be the lattice of one-parameter groups corresponding
	to the torus that acts on $X$ and let 
	$\Sigma$ be the fan in $N_\QQ = N \otimes_{\ZZ} \QQ$ that 
	corresponds to $X$.
	Now, $\OO(X)$ is generated as a $\CC$-algebra
	by the intersection of the finitely generated semi-groups 
	$\sigma^{\vee} \cap M$, 
	where $M = \Hom(N, \ZZ)$ is the dual lattice
	to $N$, $\sigma$ is a cone in $\Sigma$ and $\sigma^\vee$ denotes
	the dual cone of $\sigma$ inside $M_\QQ = M \otimes_{\ZZ} \QQ$. As the intersection of
	the semi-groups $\sigma^{\vee} \cap M$ is the intersection of
	the convex rational polyhedral cone $\cap_{\sigma \in \Sigma} \sigma^{\vee}$ with $M$, it is a 
	finitely generated semigroup 
	(Gordon's Lemma, see e.g. \cite[Proposition~1 in \S1.2]{Fu1993Introduction-to-to}). 
	This proves the first claim.
	
	Since $X$ is normal, every local ring $\OO_{X,x}$ is integrally closed, 
	and $\OO(X) = \bigcap_{x\in X}\OO_{X,x}$. 
	Hence $\OO(X)$ is also integrally closed.
\end{proof}

The next proposition is based on the study of homogeneous 
$\GG_a$-actions on affine toric
varieties in \cite{Li2010Affine-Bbb-T-varie}. Recall that a group action 
$\nu\colon G \to \Aut(X)$
on a toric variety is called \emph{homogeneous} if the torus normalizes 
the image $\nu(G)$. 
Note that for any homogeneous $\GG_{a}$-action $\nu$ there is a well-defined 
character $\chi \colon T \to \GG_{m}$, 
defined by the formula
\[
	t\,\nu(s)\, t^{-1} = \nu(\chi(t)\cdot s) \text{ for }t \in T, s \in \CC.
\]

\begin{proposition}	
	\label{prop:QuasiAffineToric}
	Let $X$ be a $n$-dimensional 
	quasi-affine toric variety.
	If $X$ is not a torus, then there exist
	homogeneous $\GG_a$-actions
	$$ 
	\eta_1, \ldots, \eta_n \colon \GG_a \times X \to X
	$$
	such that
	the corresponding characters $\chi_1, \ldots, \chi_n$ are linearly independent.
\end{proposition}

The proof needs some preparation. Denote by
$Y$ the spectrum of $\OO(X)$. 
By Lemma~\ref{lem:ToricCoordinateRing}, the variety $Y$ is normal, 
and the faithful torus action on $X$ extends uniquely to a
faithful torus action on $Y$, by Lemma~\ref{lem:ExtensionOfAction}.

The following notation is taken from \cite{Li2010Affine-Bbb-T-varie}. 
Let $N$ be a lattice of rank $n$, $M = \Hom(N,\ZZ)$
be its dual lattice and $N_\QQ = N \otimes_{\ZZ} \QQ$, 
$M_\QQ = M \otimes_{\ZZ} \QQ$.
Thus, we have a natural pairing $M_\QQ \times N_\QQ \to \QQ$, 
$(m, n) \mapsto \langle m, n \rangle$. Let $\sigma \subset N_\QQ$
be the strongly convex polyhedral cone that describes $Y$
and let $\sigma^{\vee}_M$ be the intersection of
the dual cone $\sigma^{\vee}$ in $M_\QQ$ with $M$.
Thus $Y = \Spec R$, where
\[
	R := \CC[\sigma^{\vee}_M] = \bigoplus_{m \in \sigma^{\vee}_M} \CC \chi^m
	\subseteq \CC[M] \, .
\]
For each extremal
ray $\rho \subset \sigma$, denote by $\rho^{\bot}$
the elements $u \in M_{\QQ}$ with $\langle u, v \rangle = 0$
for all $v \in \rho$. Moreover, let $\tau_M = \rho^{\bot} \cap \sigma^{\vee}_M$ 
and let
\[
	S_\rho = \{ \, e \in M \ | \ e \not\in \sigma^{\vee}_M \, , 
	\ e+m \in \sigma^{\vee}_M \ 
	\textrm{for all $m \in \sigma^{\vee}_M \setminus \tau_M$} \, \} \, .
\]
By \cite[Remark~2.5]{Li2010Affine-Bbb-T-varie} we have $S_\rho \neq \varnothing$,
and $e + m \in S_\rho$ for all $e \in S_\rho$ and all  $m \in \tau_M$.
Let us recall the description of the homogeneous locally nilpotent derivations on $R$.

\begin{proposition}[{\cite[Lemma~2.6, Theorem~2.7]{Li2010Affine-Bbb-T-varie}}]
	\label{prop.G_a-actions_on_toric_var}
	Let $\rho$ be an extremal ray in $\sigma$ and let $e \in S_\rho$.
	Then
	\[
		\partial_{\rho, e} \colon R \to R
		\, , \quad
		\chi^m \mapsto \langle m, \rho \rangle \chi^{e+m}
	\]
	is a homogeneous locally nilpotent derivation of degree $e$,
	and every homogeneous locally nilpotent derivation of $R$ 
	is a constant multiple of some $\partial_{\rho, e}$.
\end{proposition}

\begin{proof}[Proof of Proposition~\ref{prop:QuasiAffineToric}]
	Since $X$ is not a torus, $Y$ is also not a torus. Thus $\sigma$
	contains extremal rays, say $\rho_1, \ldots, \rho_k$ and $k \geq 1$.
	Recall that associated to these extremal rays, there
	exist torus-invariant divisors $V(\rho_1), \ldots,V(\rho_k)$
	in $Y$. Again,
	since $X$ is not a torus, one of these divisors does intersect $X$. 
	Let us assume that $\rho = \rho_1$
	is an extremal ray such that $V(\rho) \cap X$ is non-empty. 
	Then using the orbit-cone correspondence, one
	can see that $Y \setminus X$ is contained
	in the union $Z = \bigcup_{i=2}^k V(\rho_i)$, see 
	\cite[\S3.1]{Fu1993Introduction-to-to}. Let $e \in S_\rho$ 
	be fixed.
	We claim that the $\GG_a$-action on $Y$ associated to
	the locally nilpotent derivation $\partial_{\rho, e+m'}$
	of Proposition~\ref{prop.G_a-actions_on_toric_var}
	fixes $Z$ for all $m' \in \tau_M \setminus \bigcup_{i \geq 2} \rho_i^{\bot}$.
	
	Let us fix $m' \in \tau_M$ with $\langle m', v \rangle > 0$
	for all $v \in \bigcup_{i \geq 2} \rho_i$. Note that the fixed point set
	of the $\GG_a$-action on $Y$ corresponding to 
	$\partial_{\rho, e+m'}$
	is the zero set of the ideal generated by the image of $\partial_{\rho, e+m'}$.
	The divisor $V(\rho_i)$ is the zero set of the kernel of 
	the canonical $\CC$-algebra surjection
	\[
		p_i \colon
		\CC[\sigma^{\vee}_M] \to \CC[\sigma^{\vee}_M \cap \rho_i^{\bot}] \, , 
		\quad 
		\chi^m \mapsto \left\{\begin{array}{lr}
						\chi^m , & \textrm{if $m \in \rho_i^{\bot}$} \\
						0, & \textrm{otherwise}
					\end{array}\right. \, ,
	\] 
	see \cite[\S3.1]{Fu1993Introduction-to-to}.
	Thus we have to prove that for all $i = 2, \ldots, k$ the composition
	\[
		\CC[\sigma^{\vee}_M] \stackrel{\partial_{\rho, e+m'}}{\longrightarrow} 
		\CC[\sigma^{\vee}_M]
		\stackrel{p_i}{\longrightarrow} \CC[\sigma^{\vee}_M \cap \rho_i^{\bot}]
	\]
	is the zero map. Since, by definition, $\partial_{\rho, e+m'}$ vanishes 
	on $\tau_M = \rho^\bot \cap \sigma^{\vee}_M$, we have only
	to show that for all $m \in \sigma^{\vee}_M \setminus \tau_M$
	the following holds:
	\[
		\langle e + m' + m, v \rangle > 0 \quad
		\textrm{for all $v \in \rho_i$, $i=2, \ldots, k$.}
	\]
	This is satisfied, because $\langle m', v \rangle > 0$ and
	$\langle e+m, v \rangle \geq 0$ (note that $e \in S_\rho$
	implies $e+m \in \sigma^{\vee}_M$). This proves the claim.
	
	Since $\tau_M$ spans a hyperplane in $M$ and $e \not\in \tau_M$, we can 
	choose 
	$m'_1, \ldots, m'_n \in \tau_M \setminus \bigcup_{i \geq 2} \rho_i^{\bot}$ 
	such that $e+m'_1, \ldots, e+m'_n$ are linearly independent in $M_\QQ$.
	Hence, the homogeneous locally nilpotent derivations
	\[
		\partial_{\rho, e+ m'_i} \, , \quad \textrm{$i=1, \ldots, n$}
	\]
	define $\GG_{a}$-actions on $Y$ that fix $Z$ and thus restrict
	to $\GG_a$-actions on $X$. Moreover, the 
	character of $\partial_{\rho, e+ m'_i}$
	is $\chi_i = \chi^{e+m'_i}$. In particular, $\chi_1, \ldots, \chi_n$
	are linearly independent, 
	finishing the proof of Proposition~\ref{prop:QuasiAffineToric}.
\end{proof}

\subsection{Some topological ingredients}
For the convenience of the reader, 
we insert the following well-known statement.

\begin{lemma}
	\label{lem:FiniteGeneration}
	For a complex variety $X$, the rational singular (co)homology groups are finitely generated.
\end{lemma}

\begin{proof}
	Using the universal coefficient theorem for cohomology, it 
	is enough to prove this for the homology groups. If $X$ is affine,
	then $X$ has the homotopy type of a finite CW-complex 
	(see \cite{Ka1979Homotopy-propertie} or
	\cite[Theorem~1.1]{HaMi1997Deformation-retrac}), and hence all 
	homology groups are finitely generated. Since every
	variety can be covered by finitely many
	open affine subvarieties and since intersections of open affine subvarieties
	are again affine, the lemma follows from the Mayer-Vietoris exact sequence. 
\end{proof}

For a variety $X$, the Euler characteristic is defined by
\[
	\chi(X) = \sum_{i \geq 0} (-1)^i \dim_\QQ H^i(X, \QQ) \, ,
\]
where $H^i(X, \QQ)$ denotes the $i$-th singular cohomology group with rational
coefficients. We will use the following properties of the Euler characteristic, 
see \cite[Appendix]{KrPo1985Semisimple-group-a}.

\begin{lemma}
\label{lem:Properties_EulerChar}
	The Euler characteristic has the following properties.
\begin{enumerate}
	\item If $X$ is a variety and $Y \subseteq X$ a closed subvariety, 
		then $\chi(X) = \chi(Y) + \chi(X \setminus Y)$.
	\item If $X \to Y$ is a fiber bundle which is locally trivial in the \'etale topology with fiber $F$,
		then $\chi(X) = \chi(Y) \chi(F)$.
\end{enumerate}
\end{lemma}


\subsection{Results on the fixed point variety}
If $G$ is a group, acting on a variety $X$,
then we denote by $X^G$ the fixed point variety of $X$ by $G$. 

The first result gives us a criterion for the existence of fixed points for a $p$-group action.

\begin{proposition} 
\label{prop:EulerChar}
Let $G$ be a finite $p$-group for a prime $p$
and let $X$ be a quasi-projective $G$-variety. 
If $p$ does not divide the Euler characteristic $\chi(X)$, then 
the fixed point variety $X^G$ is non-empty.
\end{proposition}

\begin{proof}
	Assume that $X^G$ is empty, i.e. every $G$-orbit has cardinality $p^k$
	for some $k > 0$. We prove by induction on the dimension of $X$
	that $p$ divides $\chi(X)$.
	Since $X$ is quasi-projective, the same is true for the smooth locus $X^{\textrm{sm}}$
	and thus the geometric quotient 
	$\pi \colon X^{\textrm{sm}} \to X^{\textrm{sm}} /G$ exists, see 	
	\cite[Theorem~4.3.1]{Bi2002Quotients-by-actio}.
	By generic smoothness
	\cite[Corollary~10.7, Chp. III]{Ha1977Algebraic-geometry} 
	there exists an open dense subset $U$ in $X / G$ 
	such that $q := \pi |_{\pi^{-1}(U)} \colon \pi^{-1}(U) \to U$ 
	is \'etale. Since $\pi \colon X \to X / G$ is finite, $q$
	is also finite and thus $q$ is an \'etale locally trivial fiber bundle, 
	see \cite[Corollaire~5.3]{1971Revetements-etales}.
	Since each fiber of $q$ is a $G$-orbit, it follows by
	Lemma~\ref{lem:Properties_EulerChar} that the Euler characteristic
	of $\pi^{-1}(U)$ is divisible by $p$. By assumption, $G$ acts without
	fixed point on $X \setminus \pi^{-1}(U)$ and thus by induction 
	hypothesis, $\chi(X \setminus \pi^{-1}(U))$ is divisible by $p$.
	Using
	\[
		\chi(X) = \chi(\pi^{-1}(U)) + \chi(X \setminus \pi^{-1}(U))
	\] 
	(see Lemma~\ref{lem:Properties_EulerChar})
	we get that $p$ divides $\chi(X)$.
\end{proof}

The second result is a consequence of a theorem of Fogarty
\cite{Fo1973Fixed-point-scheme}. 

\begin{proposition}
\label{prop:Fogarty}
Let $G$ be a reductive group acting on a variety $X$. 
Assume that $X$ is smooth at some point $x \in X^G$.
Then $X^G$ is smooth at $x$ and the tangent space
satisfies $T_x (X^G) = (T_x X)^G$.
\end{proposition}

\begin{proof} Let us denote by $X^{(G)}\subseteq X$ the largest closed 
subscheme which is fixed under $G$,  see \cite[\S2]{Fo1973Fixed-point-scheme}) 
for details. It then follows that $X^{G}=(X^{(G)})_{\text{\it red}}$. For $x \in X$ 
we denote by $C_{x}X$ the tangent 
cone in $x$, i.e. $C_{x}X = \Spec\grad\OO_{X,x}$ where 
$\grad\OO_{X,x}:=\bigoplus_{i\geq 0}\mm_{x}^{i}/\mm_{x}^{i+1}$ is
the associated graded algebra with respect to the maximal ideal 
$\mm_{x}\subseteq \OO_{X,x}$. By definition, 
there is a closed immersion $\mu_{x}\colon C_{x}X \into T_{x}X$, and $X$ is 
smooth at $x$ if and only if $\mu_{x}$ is an isomorphism.
If $x \in X^{G}$ is a fixed point we obtain the following commutative diagram of schemes where all morphisms are closed immersions
\[
 	\xymatrix{
		C_{x}(X^{(G)}) \ar[r]^-{\subseteq} 
		\ar[d]_-{\subseteq} & T_{x} ( X^{(G)})
		\ar[d]_-{\subseteq} \\
		(C_{x}X)^{(G)} \ar[r]^-{\subseteq} &
		(T_{x} X)^{(G)} \, .
	}
\]
It is shown in
\cite[Theorem~5.2]{Fo1973Fixed-point-scheme} that for a reductive group $G$ we have $C_{x}(X^{(G)}) = (C_{x}X)^{(G)}$. If $X$ is smooth at $x$ we get 
$(C_{x}X)^{(G)} = (T_{x}X)^{(G)}$. Hence all morphisms in the diagram above are isomorphisms.
In particular $X^{(G)}$ is smooth at $x$ and thus $X^G = (X^{(G)})_{\text{\it red}}$ is smooth at $x$. Moreover, 
we get $T_{x} ( X^{G}) = (T_{x} X)^G$. 
\end{proof}

\begin{remark}\label{rem: dimension estimate}
Assume that $(\ZZ/p\ZZ)^{n}$ acts faithfully on a smooth quasi-projective 
variety $X$. If $p$ does not divide $\chi(X)$, then $\dim X \geq n$.

In fact, by Proposition~\ref{prop:EulerChar} there is a fixed point $x \in X$, and the action of $(\ZZ/p\ZZ)^{n}$ 
on the tangent space $T_{x}X$ is faithful \cite[Lemma~2.2]{KrSt2013On-automorphisms-o}, hence $n \leq \dim T_{x}X = \dim X$.
\end{remark}

\ps
\section{Proof of Theorem~\ref{thm:MainThm}
and Theorem~\ref{thm:Centralizer}}

Let us introduce the following terminology for this section.
Let $X$ be a variety and let $M \subseteq \Aut(X)$ be a subset. A map
$\eta \colon M \to Z$ into a variety $Z$ is called \emph{regular}  
if for every morphism $\nu \colon A \to \Aut(X)$ with image in $M$, the
composition $\eta \circ \nu \colon A \to Z$ is a morphism of varieties.

\subsection{Semi-invariant functions}

\begin{lemma}
\label{lem:Semi-Invariants}
Let $X$ be an irreducible normal variety, and let 
$f\in\OO(X)$ be a non-constant function such that the zero set $Z:=\VVV_{X}(f)\subset X$ is an irreducible hypersurface. Let $G \subseteq \Aut(X)$ be a connected subgroup which stabilizes $Z$.
Then we have the following.
\begin{enumerate}[label=$(\arabic*)$]
	\item \label{Eq:propery1}
	The function $f$ is a $G$-semi-invariant with character 
	$\chi\colon G \to \Cst$, i.e. $f(gx) = \chi(g)^{-1}\cdot f(x)$ 
	for all $x \in X$ and $g\in G$.
	\item \label{Eq:propery2} 
	The character $\chi \colon G \to \CC^\ast$ is a regular map. 
\end{enumerate}
\end{lemma}

For the proof we need the following description of the invertible functions on a product variety 
due to Rosenlicht \cite[Theorem~2]{Rosenlicht:1961aa}. We denote for any variety $X$
the group of invertible functions on $X$ by $\OO(X)^\ast$. 

\begin{lemma}
	\label{lem:Units_of_Prod}
	Let $X_1$, $X_2$ be irreducible varieties. Then 
	$\OO(X_1 \times X_2)^\ast =\OO(X_1)^\ast \cdot \OO(X_2)^\ast$.
\end{lemma}


\begin{proof}[Proof of Lemma~\ref{lem:Semi-Invariants}]
	\ref{Eq:propery1}
	Since $X$ is normal, the local ring $R = \OO_{X, Z}$ is a discrete 
	valuation ring.
	Let $\mm$ be the maximal ideal of $R$.
	By assumption, $f R = \mm^k$ for some $k > 0$. 
	Since $\mm$ is stable under $G$,
	the same is true for $\mm^k$. Hence, for every $g \in G$, 
	there exists a unit $r_{g} \in R^\ast$ such that $gf = r_{g}\cdot f$ in $R$.
	Since $f$ and $gf$ have no zeroes in $X \setminus Z$, it follows that $r_{g}$ 
	is regular and nonzero in $X\setminus Z$.
	Moreover, the open set where $r_{g}\in R$ is defined and nonzero meets 
	$Z$, hence $r_{g}$ is a regular invertible function
	on $X$. Consider the map 
	\[
		\chi \colon G \to \OO(X)^\ast \, ,  \  g \mapsto r_g \, .
	\]
	For all $x \in X \setminus Z$, $g \in G$ we get
	$f(g x) = \chi(g)(x)^{-1} f(x)$, and $f(g x)$, $f(x)$ are both nonzero. Since
	for each morphism $\nu \colon A \to \Aut(X)$ with image in $G$, the map
	$\tilde\nu\colon A \times X \to X$, $(a, x) \mapsto \nu(a)(x)$ is a morphism, we see that
	\[
		A \times (X \setminus Z) \to \CC^\ast \, , \ 
		(a, x) \mapsto \chi(\nu((a))(x) = f(x) \cdot f(\tilde\nu(a,x))^{-1}
	\]
	is a morphism. If $A$ is irreducible, then by Lemma~\ref{lem:Units_of_Prod}
	there exist invertible functions $q \in \OO(A)^\ast$ and $p \in \OO(X \setminus Z)^\ast$
	such that $\chi(\nu((a))(x) = q(a) p(x)$. If, moreover, $\nu(a_0) = e \in G$ for some $a_0 \in A$,
	then 
	\[
		1 = r_e(x) = \chi(\nu(a_0))(x) = q(a_0) p(x) \quad \textrm{for all $x \in X \setminus Z$} \, ,
	\]
	i.e. $p$ is a constant invertible function. 
	In this case, the composition $\chi \circ \nu \colon A \mapsto \OO(X)^\ast$
	has image in $\CC^\ast$.
	Since $G$ is connected, this implies that the whole image of $\chi$ lies in $\CC^\ast$.
	\ps
	\ref{Eq:propery2} 
	Choose $x_0 \in X \setminus Z$. As before,
	for each morphism $\nu\colon A \to \Aut(X)$ with image in $G$, the map
	\[
		A \to \CC^\ast, \quad a \mapsto \chi(\nu(a)) = f(x_0)\cdot 
		f(\nu(a)(x_0))^{-1}
	\]
	is also a morphism.
\end{proof}

\begin{lemma}
	\label{lem:InvariantHypersurfaces}
	Let $X$ be an irreducible normal variety, and let $G \subseteq \Aut(X)$ be a connected
	subgroup. Assume that  $f_{1},\ldots,f_{n} \in \OO(X)$ have
	the following properties.
	\begin{enumerate}[label=$(\arabic*)$]
		\item $Z_{i}:=\VVV_{X}(f_{i})$, $i=1,\ldots,n$, 
		are irreducible $G$-invariant hypersurfaces;
		\item $\bigcap_{i}Z_{i}$ contains an isolated point.
	\end{enumerate}
	If $\chi_{i}\colon G \to \Cst$ is the character of $f_{i}$ (see Lemma~\ref{lem:Semi-Invariants}), 
	then
	\[
		\chi := (\chi_1, \ldots, \chi_n) \colon G \to (\CC^\ast)^n
	\] 
	is a regular homomorphism with finite kernel.
\end{lemma}

\begin{proof}
Let $G$ act on $\AA^n$ by
\[
	g (a_1, \ldots, a_n) := 
	(\chi_1(g)^{-1}\cdot a_1, \ldots, \chi_n(g)^{-1}\cdot a_n).
\]
Then the map $f := (f_1,\ldots,f_n) \colon X \to \AA^n$
is $G$-equivariant. 
Let $Y \subseteq \AA^n$ be the closure of $f(X)$.
By assumption, $f^{-1}(0) = \bigcap_i Z_i$ 
contains an isolated point, hence $f\colon X \to Y$
has finite degree, i.e. the field extension 
$\CC(X) \supset \CC(Y)$ is finite. This implies that the kernel 
$K$ of $\chi \colon G \to (\CC^\ast)^n$ is finite, 
because $K$ embeds into $\Aut_{\CC(Y)}(\CC(X))$. 
By Lemma~\ref{lem:Semi-Invariants}, $\chi$ is regular.
\end{proof}

\subsection{Another centralizer result}
For an irreducible normal variety $X$, 
we denote by $\CH^1(X)$ the first Chow group, i.e.
the free group of integral Weil divisors modulo linear equivalence (see \cite[\S6, Chp. II]{Ha1977Algebraic-geometry}).
\begin{proposition}
	\label{prop:CentralizerOtherVersion}
	Let $X$ be an irreducible
	normal variety of dimension $n$ such that $\CH^1(X)$ is finite.
	Assume that for a prime $p$ the group 
	$(\ZZ / p \ZZ)^n$ acts faithfully on $X$ with a (not necessarily unique) 
	fixed point $x_0$ which is a smooth point of $X$.
	Then $G := \Cent_{\Aut(X)}((\ZZ / p\ZZ)^n)$ 
	is a closed subgroup of $\Aut(X)$ and the identity component $G^\circ$ is a closed
	torus of dimension $\leq n$.
\end{proposition}

\begin{proof}
	By \cite[Lemma~2.2]{KrSt2013On-automorphisms-o} we get 
	a faithful representation of $\Zpn$ on $T_{x_{0}}X$, and thus we
	can find generators $\sigma_1, \ldots, \sigma_n$ 
	such that $(T_{x_0} X)^{\sigma_i}\subset T_{x_{0}}X$ is a hyperplane for each $i$, 
	and that $(T_{x_0} X)^{(\ZZ / p\ZZ)^n} = 0$.
	By Proposition~\ref{prop:Fogarty}, the hypersurface $X^{\sigma_{i}}\subset X$ 
	is smooth at $x_{0}$, 
	with tangent space
	$T_{x_{0}}(X^{\sigma_{i}}) = (T_{x_{0}}X)^{\sigma_{i}}$.
	Hence there is a unique irreducible hypersurface
	$Z_i \subseteq X$ that is contained in $X^{\sigma_i}$ and 
	contains $x_0$; thus $Z_{i}$ is $G^\circ$-stable.
	Moreover, since $(T_{x_0} X)^{(\ZZ / p\ZZ)^n} = 0$, it follows that 
	${x_0}$ is an isolated point of $\bigcap_{i} Z_{i}$.
	Since a multiple of $Z_{i}$ is zero in $\CH^{1}(X)$, 
	there exist $G^\circ$-semi-invariant 
	functions $f_{i}\in\OO(X)$ such that $\VVV_{X}(f_{i})=Z_{i}$ 
	(Lemma~\ref{lem:Semi-Invariants}), and the 
	corresponding characters $\chi_{i}$ define a
	regular homomorphism
	\[
		\chi=(\chi_{1},\ldots,\chi_{n}) \colon G^\circ \to (\CC^\ast)^n
	\]
	with a finite kernel 
	(Lemma~\ref{lem:InvariantHypersurfaces}). 
	It follows that $\dim G^{\circ}\leq n$. Indeed,
	if $\nu \colon A \to \Aut(X)$ is an injective morphism with image in 
	$G^\circ$, then 
	$\chi\circ\nu\colon A \to (\Cst)^n$ is a morphism with
	finite fibers, and so $\dim A \leq n$. This implies,
	by Theorem~\ref{thm:CharacterizationAlgGrps}, that
	$G^\circ \subseteq \Aut(X)$ is an algebraic subgroup and that 
	$\chi$
	is a homomorphism of algebraic groups with a finite kernel. 
	Hence $G^\circ$ is a torus. Since $G$ is closed in $\Aut(X)$
	the same holds for $G^\circ$, see Proposition~\ref{prop:Id_component}.
\end{proof}

\subsection{Proof of Theorem~\ref{thm:Centralizer}}
	Now we can prove Theorem~\ref{thm:Centralizer} which has the same conclusion as 
	the proposition above, but under different assumptions.
	We have to show that the assumptions of
	Proposition~\ref{prop:CentralizerOtherVersion} are satisfied.
	Since $X$ is smooth, it follows that 
	$\CH^1(X)\simeq \Pic(X)$ is finite.
	Proposition~\ref{prop:EulerChar} implies that 
	the fixed point variety  $X^{\Zpn} \subseteq X$ is non-empty.
	Now the claims follow from Proposition~\ref{prop:CentralizerOtherVersion}.
\qed

\subsection{Images of maximal tori under group isomorphisms}

	\begin{proposition}
		\label{prop:MainProp}
		Let $X$ and $Y$ be irreducible quasi-projective varieties such that 
		$\dim Y \leq \dim X=:n$. Assume that the following conditions are satisfied:
		\begin{enumerate}[label = $(\arabic*)$]
		\item $X$ is quasi-affine and toric;	
		\item $Y$ is smooth, 
		$\chi(Y) \neq 0$, and $\Pic(Y)$ is finite. 
		\end{enumerate} 
		If $\theta\colon \Aut(X)\simto\Aut(Y)$ is an isomorphism, then $\dim Y=n$, and
		$\theta(T)^\circ$ is a closed torus of dimension $n$ in $\Aut(Y)$ for
		each $n$-dimensional torus $T \subseteq \Aut(X)$.
	\end{proposition}

	\begin{proof}
	Let $\theta \colon \Aut(X) \to \Aut(Y)$ be an isomorphism. Since
	$\chi(Y) \neq 0$ it follows that there is a prime $p$ that does not divide 
	$\chi(Y)$.
	
	Let $n = \dim X$ and denote by 
	$T \subset \Aut(X)$ a torus of dimension $n$. 
	We have that  $T = \Cent_{\Aut(X)}(T)$ 
	(Lemma~\ref{lem:centralizer_of_torus}), and thus $\theta(T)$
	is a closed subgroup of $\Aut(Y)$.
	Let  $\mu_p \subset T$ be the subgroup generated by the elements of order 
	$p$, and let
	$G :=  \Cent_{\Aut(Y)}(\theta(\mu_p))$ which is closed in $\Aut(X)$.
	Note that $\theta(T) \subseteq G$ and that $\dim Y = n$ by Remark~\ref{rem: dimension estimate}.
	Now Theorem~\ref{thm:Centralizer}
	implies that $G^\circ$ is a closed torus of dimension $\leq n$ in $\Aut(Y)$, and 
	by Proposition~\ref{prop:Id_component} and Theorem~\ref{thm:CharacterizationAlgGrps}, 
	$\theta(T)^\circ$ is a closed connected algebraic subgroup of $G^\circ$.
	
	In order to prove that $\dim \theta(T)^\circ \geq n$
	we construct closed subgroups
	$\{1\} = T_{0} \subset T_{1} \subset T_{2}
	\subset\cdots\subset T_{n}=T$ 
	with the following properties:
	\begin{enumerate}[label=$(\roman*)$]
		\item $\dim T_{i}=i$ for all $i$;
		\item $\theta(T_{i})$ is closed in $\theta(T)$ for all $i$.
	\end{enumerate}
	It then follows that $\theta(T_{i})^{\circ}$ is a connected algebraic subgroup of 
	$\theta(T)^{\circ}$. 
	Since the index of $\theta(T_{i})^{\circ}$ in $\theta(T_{i})$ is 
	countable (Proposition~\ref{prop:Id_component}),
	but the index of $T_{i}$ in $T_{i+1}$ is not countable, we see that 
	$\dim\theta(T_{i+1})^{\circ}>\dim\theta(T_{i})^{\circ}$, 
	and so $\dim \theta(T)^{\circ}\geq n$.
	\ps
	(a) Assume first that $X$ is a torus. Then $\Aut(X)$ contains a copy 
	of the symmetric groups $\Sn$, and we can find
	cyclic permutations $\tau_{i}\in\Aut(X)$ such that 
	$T_{i}:=\Cent_{T}(\tau_{i})$ is a closed subtorus of dimension $i$, and
	$T_{i}\subset T_{i+1}$ for all $0 < i<n$. 
	It then follows that $\theta(T_{i}) = \Cent_{\theta(T)}(\theta(\tau_{i}))$ is closed
	in $\theta(T)$, and we are done.
	\ps
	(b) Now assume that $X$ is not a torus. 
	By Proposition~\ref{prop:QuasiAffineToric}
	there exist one-dimensional unipotent subgroups $U_1, \ldots, U_n$ of 
	$\Aut(X)$ normalized by $T$ such that the corresponding characters 
	$\chi_1, \ldots, \chi_n \colon T \to \CC^\ast$ are linearly independent. Since 
	\[
		\ker(\chi_i) = \{t \in T \mid
		t \circ u_i \circ t^{-1} = u_i \textrm{ for all } u_i \in U_i \}
		= \Cent_T(U_i)
	\]
	it follows that 
	\[
		T_i := \bigcap_{k=1}^{n-i} \ker(\chi_k) = 
		\Cent_T(U_1 \cup \cdots \cup U_{n-i}) 
		\subseteq T
	\]
	is a closed algebraic subgroup of $T$ of dimension $i$. 
	It follows that the image 
	$\theta(T_{i}) = \Cent_{\theta(T)}(\theta(U_{1})\cup\cdots\cup\theta(U_{n}))$ 
	is closed in $\theta(T)$, and the claim follows also in this case.
	\end{proof}
	
	\subsection{Proof of Theorem~\ref{thm:MainThm}}
	Using Proposition~\ref{prop:MainProp}, it is enough to show 
	that a smooth toric variety $Y$ with finite (and hence trivial) Picard group is quasi-affine.
	
	
	For proving this, let $\Sigma \subset N_\QQ = N \otimes_\ZZ \QQ$ 
	be the fan that describes $Y$ where $N$ is a lattice of rank $n$. Let $N' \subseteq N$
	be the sublattice spanned by $\Sigma \cap N$ and let $Y'$ be the toric variety corresponding
	to the fan $\Sigma$ in $N'_\QQ = N' \otimes_{\ZZ} \QQ$. It follows from 
	\cite[p.~29]{Fu1993Introduction-to-to} that 
	\[
		Y \simeq Y' \times (\CC^\ast)^k
	\]
	where $k = \rank N / N'$. Thus $Y'$ is a smooth toric variety with trivial Picard group.
	Hence it is enough to prove that $Y'$ is quasi-affine and therefore we can assume $k = 0$, i.e.
	$\Sigma$ spans $N_\QQ$. By \cite[Proposition in \S3.4]{Fu1993Introduction-to-to} we get
	\[
		0 = \rank \Pic(Y) = d - n
	\]
	where $d$ is the number of edges in $\Sigma$. Let $\sigma \subset N_\QQ$ be the convex cone
	spanned by the edges of $\Sigma$ and let $\sigma^\vee$ denote the dual 
	cone of $\sigma$ in $M_\QQ = M \otimes_\ZZ \QQ$ where $M = \Hom(N, \ZZ)$. 
	Since $d = n$, the edges of $\Sigma$ are linearly independent in $\NN_\QQ$ and thus
	 $\sigma$ is a simplex. From the inclusion of the cones of $\Sigma$ in $\sigma$
	we get a morphism $f \colon Y \to \Spec \CC[\sigma^\vee \cap M]$ 
	by \cite[\S1.4]{Fu1993Introduction-to-to},
	and since each cone in $\Sigma$ is a face of $\sigma$ it is locally an open immersion
	(see \cite[Lemma in \S1.3]{Fu1993Introduction-to-to}). This implies that $f$
	is quasi-finite and birational and thus by Zariski's Main Theorem 
	\cite[Corollaire~4.4.9]{Gr1961Elements-de-geomet-III}
	it is an open immersion.
	\qed

\section{Proof of the Main Theorem}

\subsection{A first characterization}

\begin{proposition}\label{prop:XandCnIsomorphic}
Let $X$ be an irreducible quasi-affine variety. If 
$\Aut(\AA^n) \stackrel{\sim}{\to} \Aut(X)$ is an isomorphism  
that maps an $n$-dimensional torus 
in $\Aut(\AA^n)$ to an algebraic subgroup, then $X \simeq \AA^n$ as a variety.
\end{proposition}
 
\begin{proof}
Since all $n$-dimensional tori in $\Aut(\AA^n)$ are conjugate 
(see \cite{Bi1966Remarks-on-the-act}), all
$n$-dimensional tori are sent to algebraic subgroups of $\Aut(X)$ via $\theta$.
The standard maximal torus $T$ in $\Aut(\AA^n)$ 
acts via conjugation on the subgroup of standard translations $\Tr \subset \Aut(\AA^n)$
with a dense orbit $O \subset T$ and thus we get $\Tr = O \circ O$.

This implies that $S:=\theta(T)$ acts on $U:=\theta(\Tr)$ via conjugation
and we get $U = \theta(O) \circ \theta(O)$. Hence, for fixed $u_0 \in \theta(O) \subset U$ 
the morphism
\[
	S \times S \to \Aut(X) \, , \quad 
	(s_1, s_2) \mapsto s_1 \circ u_0 \circ s_1^{-1} \circ s_2 \circ  u_0 \circ s_2^{-1}
\]
has image equal to $U$. Now it follows
from Theorem~\ref{thm:CharacterizationAlgGrps}
that $U$ is a closed (commutative) algebraic subgroup of $\Aut(X)$ 
with no nontrivial element of finite order, hence
a unipotent subgroup.

We claim that $U$ has no non-constant 
invariants on $X$. Indeed, let $\rho \colon \GG_a \times X \to X$
be the $\GG_a$-action on $X$ coming from a nontrivial element of 
$U$. If $f \in \OO(X)^{U}$ is a $U$-invariant, then 
it is easy to see that 
\[\tag{$*$}
\rho_{f}(s,x):=\rho(f(x)\cdot s,x)
\]
is a $\GG_{a}$-action commuting with $U$. Since $U$ is self-centralizing, we 
see that $\rho_{f}(s) \in U$ for all $s \in \GG_{a}$. Moreover, formula $(*)$ shows that
for every finite dimensional subspace $V \subset \OO(X)^{U}$ the map $V \to U$, $f\mapsto \rho_{f}(1)$,
is a morphism which is injective. Indeed, 
$\rho_{f}(1) = \rho_{h}(1)$ implies that $\rho(f(x),x) = \rho(h(x),x)$ for all $x \in X$, 
hence $f(x) = h(x)$ for all $x \in X \setminus X^{\rho}$. It follows that $\OO(X)^{U}$ is finite-dimensional. Since $X$ is irreducible, 
$\OO(X)^{U}$ is an integral domain and hence equal to $\CC$, as claimed.

Since $X$ is irreducible and quasi-affine, the unipotent group $U$ has a 
dense orbit which is closed, and so
$X$ is isomorphic to an affine space $\AA^{m}$. 
Since $m$ is the maximal number such that there exists a faithful action of $(\ZZ/2\ZZ)^{m}$ on $\AA^{m}$ 
(see Remark~\ref{rem: dimension estimate}), we finally get $m=n$.
\end{proof}

If $X$ is an affine variety, then $X$ has the structure of a so-called affine ind-group, see e.g. 
\cite{Ku2002Kac-Moody-groups-t, St2013Contributions-to-a, FuKr2017On-the-Geometry-of} for more
details.
The following result is a special case of \cite[Theorem~1.1]{Kr2017Automorphism-Group}. It is an immediate
consequence of Proposition~\ref{prop:XandCnIsomorphic} above, because a homomorphism of affine
ind-groups sends algebraic groups to algebraic groups.

\begin{corollary}
	Let $X$ be an irreducible affine variety. 
	If there is an isomorphism $\Aut(X) \simeq \Aut(\AA^{n})$ of affine ind-groups,
	then $X \simeq \AA^n$ as a variety.
\end{corollary}
	
\begin{corollary}
	\label{cor:acyclicGeneral}
	Let $X$ be a smooth, irreducible quasi-projective variety such that 
	$\chi(X) \neq 0$ and $\Pic(X)$ is finite. 
	If there is an isomorphism $\Aut(\AA^n) \simeq \Aut(X)$ of abstract groups 
	and if $\dim X \leq n$, then $X \simeq \AA^n$ as a variety.
\end{corollary}

\begin{proof}
	Theorem~\ref{thm:MainThm} 
	shows that for an isomorphism $\theta \colon \Aut(\AA^n) \simto \Aut(X)$ and any
	$n$-dimensional torus $T \subseteq \Aut(\AA^n)$, the identity component of the image
	$S:= \theta(T)^\circ$ is a closed torus of dimension $n$ in $\Aut(X)$, 
	$\dim X = n$, and $X$ is quasi-affine.
	Thus we can apply Theorem~\ref{thm:MainThm} to $\theta^{-1} \colon \Aut(X) \simto \Aut(\AA^n)$
	and get that $\theta^{-1}(S)^\circ$ is a closed torus of dimension $n$ in $\Aut(\AA^n)$. Since
	\[
		\theta^{-1}(S)^\circ \subseteq \theta^{-1}(S) \subseteq T \, ,
	\]
	it follows that $\theta^{-1}(S) = T$, i.e. $\theta(T) = S$ is a closed $n$-dimensional torus in $\Aut(X)$.
	The assumptions of Proposition~\ref{prop:XandCnIsomorphic} 
	are now satisfied for the isomorphism $\theta \colon \Aut(\AA^n) \simto \Aut(X)$, 
	and the claim follows.
\end{proof}

\subsection{Proof of the Main Theorem}
	If~\ref{thm:AcyclicAndToric:Acyclic} holds, i.e.
	$X$ is a smooth, irreducible, quasi-projective variety of dimension $\leq n$ 
	such that $\chi(X) \neq 0$ and $\Pic(X)$ is finite, then the claim follows from
	Corollary~\ref{cor:acyclicGeneral}.
	\ps
	Now assume that the assumptions \ref{thm:AcyclicAndToric:Toric} 
	are satisfied, i.e., that
	$X$ is quasi-affine and toric of dimension $\geq n$. Let $T \subseteq \Aut(X)$ be a torus of
	maximal dimension.  We can apply 
	Theorem~\ref{thm:MainThm} to an isomorphism 
	$\theta\colon \Aut(X) \simto \Aut(\AA^{n})$ and find 
	that $S := \theta(T)^\circ \subset \Aut(\AA^{n})$ is a closed torus of dimension n.
	Since the index of the standard $n$-dimensional torus in its normalizer in $\Aut(\AA^n)$
	has finite index and since all $n$-dimensional tori in $\Aut(\AA^n)$ are conjugate 
	(see \cite{Bi1966Remarks-on-the-act}), 
	it follows that $S$ has finite index in $\theta(T)$. Hence $\theta^{-1}(S)$
	has finite index in $T$. Since $T$ is a divisible group,
	$\theta^{-1}(S) = T$ is an algebraic group.
	Thus we can apply Proposition~\ref{prop:XandCnIsomorphic}
	to the isomorphism $\theta^{-1}\colon \Aut(\AA^{n}) \simto \Aut(X)$ 
	and find that $X \simeq \AA^n$ as a variety.
\qed

\par\bigskip
\renewcommand{\MR}[1]{}
\bibliographystyle{amsalpha}

\begin{thebibliography}{Kum02}
	
	\bibitem[BB66]{Bi1966Remarks-on-the-act}
	A.~Bia{\l}ynicki-Birula, \emph{Remarks on the action of an algebraic torus on
		{$k^{n}$}}, Bull. Acad. Polon. Sci. S\'er. Sci. Math. Astronom. Phys.
	\textbf{14} (1966), 177--181. \MR{0200279}
	
	\bibitem[Bla16]{Bl2016Algebraic-Structur}
	J\'er\'emy Blanc, \emph{Algebraic structures of groups of birational
		transformations}, Proceedings of Symposia in Pure Mathematics, to appear
	(2016), \url{https://algebra.dmi.unibas.ch/blanc/articles/structurebir.pdf}.
	
	\bibitem[ByB02]{Bi2002Quotients-by-actio}
	Andrzej Bia\l~ynicki Birula, \emph{Quotients by actions of groups}, Algebraic
	quotients. {T}orus actions and cohomology. {T}he adjoint representation and
	the adjoint action, Encyclopaedia Math. Sci., vol. 131, Springer, Berlin,
	2002, pp.~1--82. \MR{1925829}
	
	\bibitem[Fil82]{Fi1982Isomorphisms-betwe}
	R.~P. Filipkiewicz, \emph{Isomorphisms between diffeomorphism groups}, Ergodic
	Theory Dynamical Systems \textbf{2} (1982), no.~2, 159--171 (1983).
	\MR{693972}
	
	\bibitem[FK17]{FuKr2017On-the-Geometry-of}
	Jean-Philippe Furter and Hanspeter Kraft, \emph{On the geometry of automorphism
		groups of affine varieties}, in preparation, 2017.
	
	\bibitem[Fog73]{Fo1973Fixed-point-scheme}
	John Fogarty, \emph{Fixed point schemes}, Amer. J. Math. \textbf{95} (1973),
	35--51. \MR{0332805}
	
	\bibitem[Ful93]{Fu1993Introduction-to-to}
	William Fulton, \emph{Introduction to toric varieties}, Annals of Mathematics
	Studies, vol. 131, Princeton University Press, Princeton, NJ, 1993, The
	William H. Roever Lectures in Geometry. \MR{1234037}
	
	\bibitem[Gro61]{Gr1961Elements-de-geomet-III}
	Alexander Grothendieck, \emph{\'{E}l{\'e}ments de g{\'e}om{\'e}trie
		alg{\'e}brique. {III}. \'{E}tude cohomologique des faisceaux coh{\'e}rents.
		{I}}, Inst. Hautes {\'E}tudes Sci. Publ. Math. (1961), no.~11, 167.
	
	\bibitem[Gro66]{Gr1966Elements-de-geomet-IV}
	\bysame, \emph{\'{E}l{\'e}ments de g{\'e}om{\'e}trie alg{\'e}brique. {IV}.
		\'{E}tude locale des sch{\'e}mas et des morphismes de sch{\'e}mas. {III}},
	Inst. Hautes {\'E}tudes Sci. Publ. Math. (1966), no.~28, 255.
	
	\bibitem[Gro71]{1971Revetements-etales}
	\bysame, \emph{Rev\^etements {\'e}tales et groupe fondamental}, Lecture Notes
	in Mathematics, Vol. 224, Springer-Verlag, Berlin-New York, 1971,
	S{\'e}minaire de G{\'e}om{\'e}trie Alg{\'e}brique du Bois Marie 1960--1961
	(SGA 1), Dirig{\'e} par Alexandre Grothendieck. Augment{\'e} de deux
	expos{\'e}s de M. Raynaud. \MR{0354651}
	
	\bibitem[Gro95]{Grothendieck:1995aa}
	\bysame, \emph{Techniques de construction et th\'eor\`emes d'existence en
		g\'eom\'etrie alg\'ebrique. {IV}. {L}es sch\'emas de {H}ilbert}, S\'eminaire
	{B}ourbaki, {V}ol.\ 6, Soc. Math. France, Paris, 1995, pp.~Exp.\ No.\ 221,
	249--276. \MR{1611822}
	
	\bibitem[Har77]{Ha1977Algebraic-geometry}
	Robin Hartshorne, \emph{Algebraic geometry}, Springer-Verlag, New
	York-Heidelberg, 1977, Graduate Texts in Mathematics, No. 52. \MR{0463157}
	
	\bibitem[HM97]{HaMi1997Deformation-retrac}
	Helmut~A. Hamm and Nicolae Mihalache, \emph{Deformation retracts of {S}tein
		spaces}, Math. Ann. \textbf{308} (1997), no.~2, 333--345. \MR{1464906}
	
	\bibitem[Kar79]{Ka1979Homotopy-propertie}
	K.~K. Kar\v{cjauskas}, \emph{Homotopy properties of algebraic sets}, Zap.
	Nauchn. Sem. Leningrad. Otdel. Mat. Inst. Steklov. (LOMI) \textbf{83} (1979),
	67--72, 103, Studies in topology, III. \MR{535582}
	
	\bibitem[Kno93]{Kn1993Uber-Hilberts-vier}
	Friedrich Knop, \emph{{\"U}ber {H}ilberts vierzehntes {P}roblem f\"ur
		{V}ariet\"aten mit {K}ompliziertheit eins}, Math. Z. \textbf{213} (1993),
	no.~1, 33--36. \MR{1217668}
	
	\bibitem[KP85]{KrPo1985Semisimple-group-a}
	Hanspeter Kraft and Vladimir~L. Popov, \emph{Semisimple group actions on the
		three-dimensional affine space are linear}, Comment. Math. Helv. \textbf{60}
	(1985), no.~3, 466--479. \MR{MR814152 (87a:14039)}
	
	\bibitem[Kra17]{Kr2017Automorphism-Group}
	Hanspeter Kraft, \emph{Automorphism groups of affine varieties and a
		characterization of affine $n$-space}, Trans. Moscow Math. Soc., to appear
	(2017), \url{http://kraftadmin.wixsite.com/hpkraft}.
	
	\bibitem[KS13]{KrSt2013On-automorphisms-o}
	Hanspeter Kraft and Immanuel Stampfli, \emph{On automorphisms of the affine
		{C}remona group}, Ann. Inst. Fourier \textbf{63} (2013), no.~3, 1137--1148.
	\MR{3137481}
	
	\bibitem[Kum02]{Ku2002Kac-Moody-groups-t}
	Shrawan Kumar, \emph{Kac-{M}oody groups, their flag varieties and
		representation theory}, Progress in Mathematics, vol. 204, Birkh{\"a}user
	Boston, Inc., Boston, MA, 2002. \MR{1923198}
	
	\bibitem[Lie10]{Li2010Affine-Bbb-T-varie}
	Alvaro Liendo, \emph{Affine {$\Bbb T$}-varieties of complexity one and locally
		nilpotent derivations}, Transform. Groups \textbf{15} (2010), no.~2,
	389--425. \MR{2657447}
	
	\bibitem[LRU17]{LiReUr2017Characterisation-o}
	Alvaro Liendo, Andriy Regeta, and Christian Urech, \emph{Characterisation of
		affine surfaces and smooth {D}anielewski surfaces}, in preparation,
	\url{http://andriyregeta.wixsite.com/homepage}, 2017.
	
	\bibitem[Ram64]{Ra1964A-note-on-automorp}
	C.~P. Ramanujam, \emph{A note on automorphism groups of algebraic varieties},
	Math. Ann. \textbf{156} (1964), 25--33. \MR{0166198}
	
	\bibitem[Reg17]{Re2017Characterization-o}
	Andriy Regeta, \emph{Characterization of $n$-dimensional normal affine
		{$SL_n$}-varieties}, \url{https://arxiv.org/abs/1702.01173}, 2017.
	
	\bibitem[Ros61]{Rosenlicht:1961aa}
	Maxwell Rosenlicht, \emph{Toroidal algebraic groups}, Proc. Amer. Math. Soc.
	\textbf{12} (1961), 984--988. \MR{0133328}
	
	\bibitem[Ryb95]{Ry1995Isomorphisms-betwe}
	Tomasz Rybicki, \emph{Isomorphisms between groups of diffeomorphisms}, Proc.
	Amer. Math. Soc. \textbf{123} (1995), no.~1, 303--310. \MR{1233982}
	
	\bibitem[Ryb02]{Ry2002Isomorphisms-betwe}
	\bysame, \emph{Isomorphisms between groups of homeomorphisms}, Geom. Dedicata
	\textbf{93} (2002), 71--76. \MR{1934687}
	
	\bibitem[Sta13]{St2013Contributions-to-a}
	Immanuel Stampfli, \emph{Contributions to automorphisms of affine spaces.},
	\url{http://edoc.unibas.ch/diss/DissB_10504}, 2013.
	
\end{thebibliography}

\providecommand{\href}[2]{#2}

\end{document}